\documentclass{eptcs}

\usepackage{amsmath,amsfonts,amssymb, graphicx}
\usepackage{todonotes}
\usepackage{amsthm}
\usepackage{tikzit, tikz-cd}
\tikzset{shorten <>/.style={shorten >=#1,shorten <=#1}}
\usetikzlibrary{decorations.markings}
\usepackage{circuitikz}
\usepackage{stmaryrd}
\usepackage{float}
\usepackage{doi}
\usepackage{breakurl}

\theoremstyle{definition}
\newtheorem{definition}{Definition}
\newtheorem{construction}[definition]{Construction}
\newtheorem{remark}[definition]{Remark}

\theoremstyle{theorem}
\newtheorem{proposition}[definition]{Proposition}

\newtheorem{example}[definition]{Example}

\newcommand{\sctikzfig}[2][.8]{\begin{center}\scalebox{#1}{\input{#2.tikz}}\end{center}}

\newcommand{\comp}{\fatsemi}

\newcommand\Para[2][{}]{\mathbf{Para}_{#1} (#2)}
\newcommand\CoPara[2][{}]{\mathbf{CoPara}_{#1} (#2)}
\newcommand\Optic[2][{}]{\mathbf{Optic}_{#1} (#2)}
\newcommand\Lens[1]{\cat{Lens}(#1)}
\newcommand\ps\Pa 
\newcommand\Sel{{\Sb}} 
\newcommand\diset[2]{\binom{#1}{#2}} 

\newcommand{\Mod}[1]{{#1}\text{-}\cat{Mod}}

\newcommand{\cat}[1]{\mathbf{#1}}
\newcommand\Set{\cat{Set}}
\newcommand\Cat{\cat{Cat}}

\newcommand\Smooth{\cat{Smooth}}

\newcommand{\action}{\bullet}
\newcommand{\ostar}{\circledast}

\newcommand{\colim}{\operatorname{colim}}
\DeclareMathOperator{\argmax}{argmax}

\DeclareMathOperator{\Ca}{\mathcal{C}}
\DeclareMathOperator{\Da}{\mathcal{D}}

\DeclareMathOperator{\Ga}{\mathcal{G}}

\DeclareMathOperator{\Ma}{\mathcal{M}}

\DeclareMathOperator{\Pa}{\mathcal{P}}

\DeclareMathOperator{\Rb}{\mathbb{R}}
\DeclareMathOperator{\Sb}{\mathbb{S}}

\newcommand{\Comega}{{\mathchoice
        {\rotatebox[origin=c]{180}{$\displaystyle \Omega$}}
        {\rotatebox[origin=c]{180}{$\textstyle \Omega$}}
        {\rotatebox[origin=c]{180}{$\scriptstyle \Omega$}}
        {\rotatebox[origin=c]{180}{$\scriptscriptstyle \Omega$}}
}}

\newcommand{\longto}{\longrightarrow}

\newcommand{\opticto}{\rightleftarrows}

\newcommand{\twoto}{\Rightarrow}

\newcommand{\narrow}[2]{\overset{#1}{#2}}
\newcommand{\nto}[1]{\narrow{#1}{\to}}
\newcommand{\nlongto}[1]{\narrow{#1}{\longto}}

\newcommand{\op}{\mathrm{op}}
\newcommand{\co}{\mathrm{co}}

\tikzstyle{circ}=[fill=white, draw=black, shape=circle]
\tikzstyle{blank}=[fill=white, draw=white, shape=circle]
\tikzstyle{none}=[fill=none, draw=none]
\tikzstyle{copy}=[fill=white, draw=black, shape=circle, minimum height=0.2cm, inner sep=0]
\tikzstyle{varCopy}=[fill=black, draw=black, shape=circle, minimum height=0.2cm, inner sep=0]
\tikzstyle{copy2}=[fill=black, draw=black, shape=circle, minimum height=0.2cm, inner sep=0]
\tikzstyle{1morph1}=[fill=white, draw=black, shape=rectangle, minimum width=1cm, minimum height=1cm]
\tikzstyle{1morph}=[fill=white, draw=black, shape=rectangle, minimum width=0.75cm, minimum height=0.75cm, inner sep=0.1cm]
\tikzstyle{2morph2}=[fill=white, draw=black, shape=rectangle, minimum width=1cm, minimum height=2cm]
\tikzstyle{2morph}=[fill=white, draw=black, shape=rectangle, minimum width=1cm, minimum height=1.25cm, inner sep=0.1cm]
\tikzstyle{nmorph}=[fill=white, draw=black, shape=rectangle, minimum height=6cm, minimum width=1cm, inner sep=0.1cm]
\tikzstyle{1state}=[fill=white, draw=black, regular polygon, regular polygon sides=3, minimum height=0.5cm, regular polygon rotate=-30]
\tikzstyle{dbox}=[fill=white, draw=black, dashed, shape=rectangle, minimum width=2cm, minimum height=1cm, inner sep=0.1cm]
\tikzstyle{vdbox}=[fill=white, draw=black, dashed, shape=rectangle, minimum width=2cm, minimum height=1.5cm, inner sep=0.1cm]
\tikzstyle{bigbox}=[fill=white, draw=black, dashed, shape=rectangle, minimum width=2cm, minimum height=4cm, inner sep=0.1cm]
\tikzstyle{2state}=[inner sep=0.05cm, fill=white, draw=black, isosceles triangle, minimum width=1.25cm, isosceles triangle apex angle=90, shape border rotate=180]
\tikzstyle{var2state}=[inner sep=0.05cm, fill=white, draw=black, isosceles triangle, minimum width=1.25cm, isosceles triangle apex angle=60, shape border rotate=180]
\tikzstyle{g2state}=[inner sep=0.05cm, fill=white, draw=black, isosceles triangle, minimum width=6cm, isosceles triangle apex angle=110, shape border rotate=180]
\tikzstyle{bigstate}=[inner sep=0.05cm, fill=white, draw=black, isosceles triangle, minimum width=3cm, isosceles triangle apex angle=110, shape border rotate=180]
\tikzstyle{bigeffect}=[inner sep=0.05cm, fill=white, draw=black, isosceles triangle, minimum width=3cm, isosceles triangle apex angle=110]
\tikzstyle{g2effect}=[inner sep=0.05cm, fill=white, draw=black, isosceles triangle, minimum width=6cm, isosceles triangle apex angle=110]
\tikzstyle{2effect}=[inner sep=0.05cm, fill=white, draw=black, isosceles triangle, minimum width=1.25cm, isosceles triangle apex angle=90]
\tikzstyle{b2effect}=[inner sep=0.05cm, fill=white, draw=black, isosceles triangle, minimum width=2cm, isosceles triangle apex angle=90]
\tikzstyle{midArrow}=[-, decoration={{markings,mark=at position .5 with {\arrow{>}}}}, postaction=decorate]

\tikzstyle{Medium box}=[fill=white, draw=black, shape=rectangle, tikzit shape=rectangle, minimum width=1.5cm, minimum height=1.5cm]
\tikzstyle{Black arrow}=[->]
\tikzstyle{Gray line}=[-, draw={rgb,255: red,191; green,191; blue,191}, line width=0.8]

\tikzstyle{arrow}=[->]

\providecommand\event{ACT 2021}
\title{
    Towards Foundations of Categorical Cybernetics
}
\author{
    Matteo Capucci, Bruno Gavranovi\'c,\\
    Jules Hedges, Eigil Fjeldgren Rischel
    \institute{MSP Group, University of Strathclyde, \'Ecosse Libre}
}
\def\titlerunning{
    Towards Foundations of Categorical Cybernetics
}
\def\authorrunning{
    M. Capucci, B. Gavranovi\'c, J. Hedges, E. Fjeldgren Rischel
}

\begin{document}
    \maketitle

    \begin{abstract}
        We propose a categorical framework for processes which interact bidirectionally with both an environment and a `controller'.
        Examples include open learners, in which the controller is an optimiser such as gradient descent, and an approach to compositional game theory closely related to open games, in which the controller is a composite of game-theoretic agents.
        We believe that `cybernetic' is an appropriate name for the processes that can be described in this framework.
    \end{abstract}

    \section{Introduction}

In this paper we propose a categorical framework for processes which interact bidirectionally with both an environment and a `controller'. Examples include open learners \cite{fong-spivak-tuyeras-backprop-as-functor, CatFoundGradDesc}, in which the controller is an optimiser such as gradient descent, and an approach to compositional game theory closely related to open games \cite{hedges_etal_compositional_game_theory} in which the controller is a composite of game-theoretic agents. We believe that `cybernetic' is an appropriate name for the processes that can be described in this framework.

A common theme of this paper is that the same categorical ingredients appear repeatedly at multiple `levels' in the construction of cybernetic systems. The most important of these is the \emph{para construction}, which yields a category $\Para{\Ca}$ of parametrised morphisms given a category $\Ca$ acted upon by a monoidal category of parameters $\Ma$.
%
%
We combine this with the optics construction \cite{ProfunctorOptics}, which is known to model bidirectional information flow in dynamical systems \cite{myers2020double} and compositional game theory \cite{bolt_hedges_zahn_bayesian_open_games}.
We argue that we can think of optics as being built of coupled parametrised and `coparametrised' morphisms.
This suggests that $\Para{-}$ and $\Optic{-}$ are related.
Putting them together, we obtain the category of parametrised optics, which is the central construction of this paper.

A parametrised optic describes open systems that have bidirectional information flow, with a control on the forward direction and an objective on the backwards direction. This pattern is ubiquitous in cybernetics: whenever an instance of the optic pattern appears, we generally find that they form a parametrised family with an objective function, where the goal of the system is to find a parameter that optimises the objective, all while interacting with its environment bidirectionally.
The most straightforward example is a component (for example a layer) of a neural network, which interacts with its neighbours via forwards and backwards passes, depending on its current parameters (weights and biases), and produces a `local' loss against which the parameters are optimised.

The final part of the construction of a class of cybernetic systems is to describe how the parameters/controls and objectives interact. For many classes of systems, they are coupled to a dynamical system that updates the parameters using an algorithm such as gradient descent.
In the case of game theory the situation is more complicated, since game-theoretic agents perform counterfactual reasoning and are classically assumed to be perfect optimisers rather than using some effective approximation of the optimum. This situation can be described using selection functions \cite{hedges_etal_selection_equilibria_higher_order_games}.
We find that there is a deep and previously unknown relationship between selection functions and optics; in particular we can build a category whose objects are selection functions and morphisms are optics, with a monoidal product describing Nash equilibria.

By combining parametrised optics and selection functions, we obtain a compositional formulation of game theory. This turns out to be more specific than open games: a pair of a parametrised optic and a selection function determines an open game functorially, but not vice versa. However each part of our construction carries a clear game-theoretic meaning, and we also fix a particular technical problem with open games involving games where one agent makes several different decisions. Overall, we claim that this variant of open games is an improvement on the usual one.

For reasons of space we only present two classes of examples: neural networks which can be presented entirely using the structure of parametrised optics, and open games which have selection functions as an additional ingredient. However, abstractly the same construction works if selection functions are replaced with any suitable lax monoidal pseudofunctor on a category of optics. Further applications of this framework are given in a companion paper \cite{other_paper}.


    \section{Para}
Throughout this section, let $(\Ma, J, \odot)$ be a strict
monoidal category and let $\Ca$ be a category.

\begin{definition}
	An $\Ma$-actegory is a category $\Ca$ together with a functor $\action : \Ma \times \Ca \to \Ca$, and natural transformations $\epsilon : J \action X \cong X$ and $\delta : (M \odot N) \action X = M \action (N \action X)$ satisfying certain coherence laws \cite[Definition 2.2]{vasilakopoulou2018enriched}.
\end{definition}

%

Actegories (note the intentional spelling change) are a vertical categorification of monoid actions.
For the purposes of this paper, $\Ma$ is to be thought of as a category of \emph{parameters} and $-\action-$ as a \emph{parametrisation operation}.
From the structure of an $\Ma$-actegory then, we can produce a category of $\Ma$-parametrised morphisms in $\Ca$:

\begin{definition}[Para]
\label{def:para}
	Let $\Ca$ be an $\Ma$-actegory.
	$\Para[\action]{\Ca}$ is the bicategory whose
	\begin{enumerate}
		\item objects are the same as $\Ca$;
		\item 1-cells $\phi \in \Para[\action]{\Ca}(X,Y)$ are given by a choice of object $M$ (the \emph{parameter}) in $\Ma$ and a choice of morphism $\phi : M \action X \to Y$ in $\Ca$, and we condense this data into the notation $\phi : X \nto{M} Y$;
		\item 1-composition of $\phi : X \nto{M} Y$ and $\psi : Y \nto{N} Z$ is given by
		\[
			\phi \comp \psi : X \nlongto{N \odot M} Z
			\quad := \quad
			(N \odot M) \action X \nlongto{\delta} N \action (M \action X) \nlongto{N \action \phi} N \action Y \nlongto{\psi} Z
		\]
		\item 2-cells $r : \phi \twoto \phi'$ from $\phi : X \nto{M} Y$ to $\phi' : X \nto{M'} Y$ are given by morphisms $r : M' \to M$ in $\Ma$ such that
		\[
			\begin{tikzcd}
				M' \action X \arrow[swap]{d}{r \action X} \arrow{dr}{\phi'}\\
				M \action X \arrow[swap]{r}{\phi} & Y
			\end{tikzcd}
		\]
		commutes.
		\item Identities and composition in the hom-cat $\Para[\action]{\Ca}(X, Y)$ are the same as $\Ma$.
	\end{enumerate}
\end{definition}

\begin{figure}
	\begin{minipage}{0.5\textwidth}
		\sctikzfig{figures/para}
	\end{minipage}\begin{minipage}{0.5\textwidth}
		\sctikzfig{figures/copara}
	\end{minipage}
	\caption{(a) A morphism in $\Para{\Ca}$ \hspace{3cm} (b) A morphism in $\CoPara{\Ca}$}
	\label{fig:para_morphism}
\end{figure}

Morphisms in $\Para{\Ca}$ are more naturally denoted using string diagrams, as shown in Figure~\ref{fig:para_morphism}. In these diagrams, objects of $\Ma$ are denoted by vertical strings, and objects of $\Ca$ by horizontal strings. Formally, these (and most other diagrams in this paper) can be considered as diagrams in a larger double category in which both $\Para{\Ca}$ and $\CoPara{\Ca}$ embed, and thus we can apply coherence for string diagrams in double categories \cite{meyers_string_diagrams_double_categories}.

\begin{figure}[h]
	\sctikzfig{figures/para_composition}
	\caption{1-cell composition in $\Para{\Ca}$}
	\label{fig:para_comp}
\end{figure}
\begin{figure}[h]
	\sctikzfig{figures/reparametrisation}
	\caption{Reparametrisation along a 2-cell $r$ in $\Para{\Ca}$}
	\label{fig:para_reparam}
\end{figure}

\begin{remark}
	The most common usage of 2-cells in $\Para{\Ca}$ is through a lift-like operation called \emph{reparametrisation}, i.e. using a map of parameters (thus in $\Ma$) to change parameter space. This means we rarely provide the codomain of a 2-cell and instead we produce it by lifting.
	Explicitly, given $r : M' \to M$ and an $M$-parametrised morphism $\phi : M \action X \to Y$, we produce an $M'$-parametrised morphism $\phi' = (r \action X) \comp \phi$ which we denote by $r^* \phi$ (an operation depicted in Figure~\ref{fig:para_reparam}), so that $r:\phi \twoto r^* \phi$ \emph{lifts} $r : M' \to M$.
	This amounts to say that $\Para{\Ca}$ is \emph{locally fibred} over $\Ma$.
\end{remark}

\begin{remark}
	There is a dual construction to that of $\Para{-}$, which we call $\CoPara{-}$, which produces a category of \emph{coparametrised} morphisms, i.e. morphisms whose \emph{co}domain is parametrised (Figure~\ref{fig:para_morphism}(b)).
	A morphism $X \to Y$ in $\CoPara{\Ca}$ consists of an object $M$ and a morphism $\phi : X \to M \bullet Y$.
	Moreover, 2-cells go in the opposite direction (again, as hinted by the graphical language): the recoparametrisation of $\phi$ along $r : M \to M'$ is a 2-morphism $\phi \twoto r_* \phi$ defined by postcomposition with $r \action Y$.
	This makes $\CoPara{\Ca}$ locally opfibred over $\Ma$.
	All things considered, we can concisely define $\CoPara{\Ca}$ as $\Para{\Ca^{\op}}^{\co\op}$.
\end{remark}

It is useful to notice that $\Para{-}$ is a functorial construction, which also enjoys a monadic structure:

\begin{proposition}\label{prop:para_monad}
	$\Para{-}$ is a monad on the category $\Mod\Ma$ of $\Ma$-actegories.
\end{proposition}
\begin{proof}
	The action of $\Ma$ on $\Para{\Ca}$ is inherited from $\Ca$ in the obvious way.
	Its unit is given by faithful, identity-on-objects functors sending a morphism $\phi$ to $(J, \phi)$.
	Its join sends a twice parametrised map $(m, n, \phi)$ to the once parametrised map $(m \odot n, \phi)$.
\end{proof}

\subsection{Monoidal structure on \texorpdfstring{$\Para{\Ca}$}{Para(C)}.}
If both $\Ma$ and $\Ca$ are additionally monoidal categories in a way that is compatible with the actegory structure, we would like to lift the monoidal structure to $\Para{\Ca}$. Unfortunately bifunctoriality of the monoidal product on $\Para{\Ca}$ requires that the monoidal product on $\Ma$ is \emph{commutative}, that is, $M \odot N = N \odot M$ on the nose rather than up to a symmetry morphism. We believe that in the general case, $\Para{\Ca}$ is symmetric monoidal as a bicategory \cite[Section 12.1]{johnson2021}. Commutative monoidal categories are rare in practice, but for presentation purposes in this paper we assume that $\Ma$ is commutative monoidal.


\begin{definition}\label{def:para_monoidal}
	A symmetric monoidal $\Ma$-actegory is an $\Ma$-actegory $\Ca$ equipped with a symmetric monoidal structure, together with a natural isomorphism $\kappa_{M,X,Y} : M \action (X \otimes Y) \cong X \otimes (M \action Y)$, satisfying coherence laws reminiscent of the laws of a costrong comonad.
\end{definition}

Given a symmetric monoidal $\Ma$-actegory, we can also define natural isomorphisms $\alpha : M \action (X \otimes Y) \cong (M \action X) \otimes Y$ and $\iota : (M \odot N) \action (X \otimes Y) \cong (M \action X) \otimes (N \action Y)$, called \emph{mixed associator} and \emph{interchanger}, respectively.

As anticipated, when $\Ca$ is a symmetric monoidal $\Ma$-actegory, then $\Para[\action]{\Ca}$ is itself a symmetric monoidal category, whose structure we still denote with $I$ and $\otimes$.
The tensor of two objects $X, Y$ is given by the same tensor $X \otimes Y$, whereas the tensor of two arrows $\phi : X \nto{M} Y$ and $\psi : U \nto{N} V$ is given by
\[
	\phi \otimes \psi : X \otimes X' \nlongto{M \odot N} Y \otimes Y'
	\quad := \quad
	(M \odot N) \action X \otimes X' \nlongto{\iota} (M \action X) \otimes (N \action Y) \nlongto{\phi \otimes \psi} Y \otimes Y'
\]
Bifunctoriality follows from the commutativity of $\Ma$.
Unitors and associator are borrowed from $\Ca$ through the immersion $\Ca \to \Para[\action]{\Ca}$.

\begin{figure}[ht]
	\sctikzfig{figures/para_tensor}
	\caption{Monoidal product in $\Para{\Ca}$}
	\label{fig:para_tensor}
\end{figure}

\begin{remark}
	It turns out that the structure of a symmetric monoidal $\Ma$-actegory $\Ca$ is equivalent to that of a strong monoidal functor $\Ma \to \Ca$.
	Given an action, this corresponding functor is $F: M \mapsto M \action I$. Given a functor, the corresponding action is $(M, X) \mapsto F(M) \otimes X$.
	We can think of the image of $- \action I$ as a `representation' of $\Ma$ inside $\Ca$.
\end{remark}


    \section{Optics}

The story of dynamical systems is grounded in the notion of bidirectionality and feedback: processes do not just flow in one direction, but often compute values \textit{going forward} and compute some notion of \textit{feedback} or \textit{rewards} going backward.
In order of increasing generality, this is captured by lenses, optics, and mixed optics. In this paper we proceed with the latter notion and invoke special cases as needed.

\begin{definition}[{Optics, \cite[Def. 2.1]{ProfunctorOptics}}]\label{def:optics}
  Given a monoidal category $\Ma$ acting on categories $\Ca, \Da$ through actions we both denote with $\action$ be two $\Ma$-actegories, we can define the category $\Optic[\action,\action]{\Ca, \Da}$ whose objects are pairs of objects of $\Ca$ and $\Da$, and whose hom-sets are defined by the following coend
  \[
    {\textstyle \Optic[\action,\action]{\Ca, \Da}(\diset{X}{X'}, \diset{Y}{Y'})} = \int^M \Ca(X, M \action Y) \times \Da(M \action Y', X')
  \]
\end{definition}

We can think of optics as a bidirectional process taking in an $X$ and computing an output $M \action Y$. Depending on the monoidal action used, this can unpack to a number of different things. Most often it is a product of an output $Y$ and the intermediate state $M$ (called the residual). Off-screen, the environment takes in the output $Y$ and returns a $Y'$. In the case of the monoidal action of the product, this $Y'$ is used together with the previously saved $M$ to produce an $X'$ (Figure~\ref{fig:optic}).

Concretely, an optic $\diset{X}{X'} \opticto \diset{Y}{Y'}$ is an equivalence class of triples $(M, v, u)$ where $M$ is an object of $\Ma$, $v : X \nto{M} Y$ is a morphism of $\CoPara{\Ca}$ and $u : Y' \nto{M} X'$ is a morphism of $\Para{\Ca}$, modulo the equivalence relation generated by $(M, v, f^* u) \sim (M', f_* v, u)$ for morphisms $f : M \to M'$ of $\Ma$.
This is a small equivalent reformulation of the usual definition of optics via a coend.
The coend in this formulation quotients out observationally indistinguishable reparametrisations, meaning that morphisms of residuals can be freely slid over back and forth between the forward and backward part.

\begin{remark}[Specializations]\label{rem:spec}
  If $\Ma = \Ca=\Da$, we speak of \textit{non-mixed optics} and we write $\Optic[\otimes]{\Ca}$, sometimes omitting the action as well.
  If $\Ca$ and $\Da$ are symmetric monoidal, then $\Optic[\action,\action]{\Ca, \Da}$ is itself symmetric monoidal.
  Also observe that $\cat{Lens(\Ca)} \cong \Optic[\times]{\Ca}$ for a cartesian monoidal category.
\end{remark}

\begin{figure}[h]
  \sctikzfig{figures/optic}
  \caption{Graphical representation of an optic $\diset{X}{X'} \opticto \diset{Y}{Y'}$. We conventionally think of the $u$ node as rotated rather than reflected, so parameters enter at the bottom (and, later, coparameters will leave at the top). For an animated version explaining the information flow, see \cite{OpticAnimation,OpticCompositionAnimation}.}
  \label{fig:optic}
\end{figure}


Under certain conditions, we have a convenient description of the collection of \textbf{states} (maps from the monoidal unit) and \textbf{costates} (maps to the monoidal unit) in the category of optics.
On the one hand, states of $\diset{X}{X'}$ in $\Optic{\Ca,\Da}$ when $\Da$ and $\Ma$ are semicartesian (hence their unit coincides with the terminal object) coincide with states of $X$ in $\Ca$:
\[
  \int^{M}\Ca(1, X \action M) \times \Da(X' \action M, 1) \cong \int^{M}\Ca(1, X \action M)
  \cong \colim_{M} \Ca(1, X \action M) \cong \Ca(1, X)
\]
Here we use first terminality of $1$ in $\Da$, then the fact that mute coends are colimits \cite[Lemma 1.2.4]{CategoriesOfOptics}, then the fact that a colimit over an index category with a terminal object is just the functor evaluated at that object, and $1$ is terminal in $\Ca$.
On the other hand, costates of $\diset{X}{X'}$ in non-mixed optics coincides with maps $X \to X'$ in $\Ca$.
\[
  \int^{M}\Ca(X, M \otimes I) \times \Ca(M \otimes I,X') \cong \int^{M}\Ca(X, M)
    \times \Ca(M,X') = \Ca(X,X')
\]
by the ninja Yoneda lemma \cite{loregian2020coend}.
We note that the case when the forward and the backward categories are different can still be made precise, but now instead of a homomorphism $X \to X'$ in $\Ca$ we have a \textbf{heteromorphism} $X \to X'$ from $\Ca$ to $\Da$ (mediated by a residual in $\Ma$).



% ...



    \section{Parametrised optics for cybernetics systems}\label{sec:para-optic}

The $\Para{-}$ construction lets us consider morphisms with ``hidden'' information, not available at the boundary of a morphism. The $\Optic{-}$ construction lets us consider bidirectional morphisms in this category, allowing for constructions involving \emph{feedback} and \emph{update}. Applying the former to the latter results in a sophisticated construction which we unpack in this section.

\paragraph{Parametrised optics.}
Given a symmetric monoidal category $\Ma$ acting on $\Ca$ and $\Da$, we have shown above how to form the category $\Optic{\Ca, \Da}$.
Moreover, since $\Ma$ is a symmetric monoidal category, it acts on itself by the monoidal product, so we can form the category $\Optic\Ma$ of non-mixed optics, which is symmetric monoidal with pointwise monoidal product.

\begin{proposition}
\label{prop:canon_optic_action}
	Let $\Ma$ be symmetric monoidal and $\Ca$ and $\Da$ be $\Ma$-actegories. There is an action of $\Optic[\odot]\Ma$ on $\Optic[\action,\action]{\Ca,\Da}$, defined on objects as
	$
		\diset{M}{M'} \ostar \diset{X}{X'} := \diset{M \action X}{M' \action X'}.
	$
\end{proposition}

We thus can form the bicategory $\Para[\ostar]{\Optic[\action,\action]{\Ca, \Da}}$ for this action.
Concretely, for objects $X,Y$ of $\Ca$ and $X',Y'$ of $\Da$, a \textbf{parametrised optic} $\diset{X}{X'} \opticto \diset{Y}{Y'}$ consists of a choice of objects $(P, Q)$ of $\Ma$, and an equivalence class of triples $(M, v, u)$ where $M$ is an object of $\Ma$, $v : P \action X \to M \action Y$ is a morphism of $\Ca$ and $u : M \bullet Y' \to Q \bullet X'$ is a morphism of $\Da$.
We depict such a morphism by a string diagram, shown in Figure~\ref{fig:para_optic}(a).

\begin{figure}[ht]
	\begin{minipage}{0.5\textwidth}
		\sctikzfig{figures/paraoptic}
	\end{minipage}\begin{minipage}{0.5\textwidth}
		\sctikzfig{figures/paraopticreparam}
	\end{minipage}
	\caption{A morphism (a) and reparametrisation (b) in $\Para[\ostar]{\Optic[\action,\action]{\Ca, \Da}}$.}
	\label{fig:para_optic}
\end{figure}

\paragraph{Cybernetic systems.}
The mathematical framework of parametrised optics extends the one already known
for dynamical systems, in which different flavours of lenses are used to
represent bidirectional information flow \cite{vagner2014algebras,
  myers2020double, cts}, in a direction first hinted at
in~\cite{fong-spivak-tuyeras-backprop-as-functor} and~\cite{Dioptics}.
Parametrised optics model cybernetic systems, namely dynamical systems
\emph{steered} by one or more agents.
Then $\ostar$ represents \emph{agency being exerted} on systems; in particular
agents act in the world through the $\Ma$-actegorical structure on $\Ca$ and
receive feedback through the $\Ma$-actegorical structure on $\Da$.%
\footnote{A happy coincidence of terminology: \emph{agents act through actions}.}
$\Optic[\odot]\Ma$ is a category representing agents, and the framework of $\Para[\ostar]{\Optic[\action,\action]{\Ca, \Da}}$ is what allows us to interpret it as such.%
\footnote{We remark agency is a property of the \textit{model} of a system and not of the system \emph{per se}.
This is reflected in the mathematics by the fact that all morphisms in
$\Para{\Ca}$ are morphisms of $\Ca$, arranged in a different way.}


Graphically, vertical wires in $\Para[\ostar]{\Optic[\action,\action]{\Ca, \Da}}$ tell us how agents interact and compose.
Reparametrisations (Figure~\ref{fig:para_optic}(b)) denote agency dynamics as
happening \emph{over} the horizontal dynamical system (which can be called an
\textbf{arena}, or a \textbf{protocol}) they act within.
This point of view seamlessly unifies parameters and coparameters of a parametrised optic with residuals inside the optic, which are now representing the \emph{private state} of agents controlling the given dynamical system.
This state ferries information between the forward and the backward part of the dynamics (playing the role of `memory') so that an agent's feedback is contingent on the action they brought about (see \cite[p.~192]{strategy}).%
\footnote{Further investigation in this direction is being pursued by the authors.
In particular, it seems profitable to adopt a notion of \emph{lax optics} in which the coend of Definition~\ref{def:optics} is replaced by a lax coend \cite[Chapter 7]{loregian2020coend}.
Residuals and slidings would then be explicitly tracked as 2-dimensional structure on $\Optic{\Ca,\Da}$, as much as parameters in $\Para{\Ca}$ are tracked in the 2-dimensional structure of that category.
Lax fubini (ibid.) gives rise to an interesting duality between residuals and
parameters, which the authors first heard of in a Zulip conversation started by
Mitchell Riley~\cite{RileyZulip}.} Furthermore, more complex protocols of
interaction can be created, for instance, using optics for the coproduct
(prisms), modeling control flow.

What is missing in the mathematical structure of $\Para{\Optic{-}}$ is a feedback mechanism, central in cybernetics.
In the next sections, we explore the current solutions adopted by the two most developed instantiations of the framework just described, namely \emph{open learners} and \emph{open games}.

\subsection{Neural Networks}
A special case of this construction is studied in \cite{CatFoundGradDesc}, in which the authors instantiate this construction in the context of machine learning, when the base is set to $\Smooth$, the category of Euclidean spaces and smooth maps.
They first start by reframing backpropagation in terms of optic composition, via the functor $R: \Smooth \to \Optic{\Smooth}$ which augments maps with their backward pass (though they use lenses, as defined in Remark~\ref{rem:spec}).
Then, using the action of $\Para{-}$ on morphisms (Proposition~\ref{prop:para_monad}), they lift this functor to its $\Para{-}$ counterpart.
We now unpack this in more detail.

Consider a morphism in $(\Rb^p, f) : \Para{\Smooth}(\Rb^n, \Rb^m)$.
It consists of a \textit{parameter space} $\Rb^p$ and a smooth function $f : \Rb^p \times \Rb^n \to \Rb^m$.
For instance, this could be the forward pass of a complex, multi-layered neural network.
It takes in a parameter value $p : \Rb^p$, an input $x : \Rb^n$ and computes a \textit{prediction} $f(p, x) : \Rb^m$.
Composition of morphisms in $\Para{\Smooth}$ reduces to composition of forward passes of neural networks, coherently tracking which subnetworks each
incoming parameter needs to be relayed to.
This becomes important as the authors apply a functor
\begin{equation}
	\label{eq:para_r}
	\Para{R} : \Para{\Smooth} \to \Para{\Optic{\Smooth}},
\end{equation}
augmenting this neural network with the \textit{backward pass}.
Acting on the morphism $f$, this results in a forward and a backward component (Figure~\ref{fig:para_optic}(a)).
The forward component consists of an input parameter space $\Rb^p$ and a morphism $v : \Rb^p \times \Rb^n \to \Rb^m \times \Rb^n$ which in addition to computing $f$ also copies the input $\Rb^n$ and saves it as the \textit{residual} used in the backward pass.
The backward pass consists of an output parameter space $\Rb^p$ (interpreted as the space of \textit{changes}), and a map $u : \Rb^n \times \Rb^m \to \Rb^n \times \Rb^p$ which takes in a \textit{change in the output}, and together with the saved residual computes a change in the parameters and change in inputs.
The latter is then subsequently used in the previous learner as its incoming change in outputs.
This complex data flow is automatic: it falls out of the machinery of both $\Para{-}$ and $\Optic{\Smooth}$.
The former deals with parameter spaces, and the latter knows how to perform backpropagation.

These parameter spaces become important as neural networks can be coherently
\textit{reparametrised} within this framework, where reparametrisations in
$\Para{\Optic{\Smooth}}$ are themselves optics.
For instance, consider the parameter port $(\Rb^p, \Rb^p)$ of a 1-cell: it takes in a parameter value and produces a \emph{change in that parameter}.
By using the optic of gradient descent we can reparametrise this learner and express the idea that a learner is \emph{minimizing} a particular objective function.

\begin{construction}[{Gradient Descent, compare \cite[Example 3.15]{CatFoundGradDesc}}]
\label{ex:grad_desc}
  Consider a lens on the base $\Ca = \Smooth$.
  Fix $\alpha : \Rb$.
  The notion of $\alpha$-gradient descent is a lens $gd_{\alpha} : \diset{\Rb}{\Rb} \opticto \diset{\Rb}{\Rb}$ whose forward map is identity, and the backward map $(p, \nabla p) \mapsto p - \alpha \nabla p$ computes a new parameter value by moving in the $\alpha$-scaled direction of the change in the parameter $\nabla p$.
  For negative choices of $\alpha$ we obtain gradient \emph{ascent}, and denote it with $ga_{\alpha}$.
\end{construction}

The archetypal example of two neural networks optimizing \textit{different goals} is Generative Adversarial Networks (GANs) \cite{GAN, WGAN}.
We hereby present the first categorical formalization of this neural network architecture and interpret its game-theoretic behaviour in terms of opposing reparametrisations: gradient \textit{descent} and gradient \textit{ascent}.%
\footnote{We focus on WGAN for simplicity, and ignore the Lipschitz regularization condition.}
A GAN is composed of two networks: a \textit{generator (g)} and a \textit{discriminator (d)} network in $\Para{\Optic{\Smooth}}$, obtained as the image under the functor $\Para{R}$ (Equation~\ref{eq:para_r}).
The discriminator (an element of $\Para{\Smooth}(\Rb^x, \Rb)$) is tasked with
assigning an $\Rb$-valued \textit{cost} to each point of a particular ``image space'' (denoted with $\Rb^x$) as a measure of how much it \textit{looks like}
it belongs to a particular predefined dataset, where low cost means realistic-looking samples.
The generator (an element of $\Para{\Smooth}(\Rb^z, \Rb^x)$) is tasked with
generating realistic-looking samples from this dataset, by taking in some
``latent vector'' in $\Rb^z$ and producing a sample in $\Rb^x$ (often called the ``fake'' sample).
The user's goal is to train the generator such that, as we vary the latent vector input $\Rb^z$, we obtain different realistic-looking samples in $\Rb^x$.
This is done via the \textit{adversarial training} regime and a carefully chosen composition of these networks (Figure~\ref{fig:gan}).
This regime feeds in real and fake samples to the discriminator, training it to distinguish between them.
At the same time, it trains the generator with an opposing goal: producing samples that \emph{fool} the discriminator, making it assign a low cost to them.

\begin{figure}[ht]
  \sctikzfig{figures/GAN}
  \caption{A generative adversarial network as a closed system.}
  \label{fig:gan}
\end{figure}

This adversarial component is in our framework captured by reparametrisation, allowing us to specify agent preferences \textit{internal to the agent}.
Consider the composite $R(g) \comp R(d)$ in $\Para{\Optic{\Smooth}}$ (middle part of Figure~\ref{fig:gan}).
At some time step $i$ this neural network takes a latent vector $z_i$ as input, produces a
\textit{fake} element $x_i$ of $\Rb^x$ using $g$, and assigns an $\Rb$-valued score to it using $d$.
In this setting we want the discriminator to assign an even higher score next time (as the sample was fake) and the generator to produce a sample which will have a lower score next time (as it is trying to fool the discriminator).
We can achieve this by reparametrising both the generator and the discriminator, but with different, ``opposing'' optics.
We reparametrise the generator with the gradient \textit{descent} (making it move in the negative direction of the gradient, minimizing the cost the discriminator assigns to the fake sample), and the discriminator with the gradient \textit{ascent} (making it move in the positive direction of the gradient, increasing the cost discriminator assigns to the fake sample).
This allows both of these agents to have unchanged external behaviour during one
time-step of updating, and use the same training signal to enact a different
update rule internally.%
\footnote{The lens costates on the right of
  Figure~\ref{fig:gan} are aptly called $dx$. Their backward maps are constant at $1 : \mathbb{R}$, the usually invisible ``initial'' factor of $1$ in backpropagation.}

Another benefit of reparametrisation is that it allows us to recast two 1-cells with the same parameter object as \textit{the same player}.
Consider the discriminator in Figure~\ref{fig:gan}.
It appears twice: once valuating samples coming from the generator ($x_i$) and
once evaluating samples coming from the actual dataset ($d_i$).
The parameter port of the parallel product (Definition~\ref{def:para_monoidal}) $R[d] \times R[d]$ takes in two parameters $\Rb^q \times \Rb^q$ as input and produces \textit{two} changes $\Rb^q \times \Rb^q$ in these parameters as output.
By first reparametrising this product using the image of the copy map $\Delta_{\Rb^q} : \Rb^q \to \Rb^q \times \Rb^q$ under the functor $R$, we make sure that the same parameter is copied to both discriminators in the forward pass, and their gradients are summed up on the backward pass.
Also called \textit{weight tying}, this couples together two different neural
networks and treats them as a single unit.




    \section{Selection relations}

In cybernetics, agents provide parameters which are evaluated by the environment through a dynamical process and fed back to them.
Crucially, agents then use this information to update the parameters they provide, closing the loop.
When agents are satisfied with the outcome of the interaction, an equilibrium is established.
In machine learning, parameter updating is explicitly modelled (Construction~\ref{ex:grad_desc}), but in game theory we directly seek the equilibrium.
We propose here a framework for equilibrium selection for parametrised optics based on a generalization of selection functions.

A selection relation, defined in \cite{HigherOrderDecisions} as ``multi-valued selection functions'', is a relation of type $\varepsilon \subseteq X \times (X \to R)$. Selection relations provide an abstraction of the ``personality'' or ``goals'' of game-theoretic agents. Specifically, we consider that a function $\varepsilon \subseteq X \times (X \to R)$ describes an agent who can make a choice from a set $X$, with outcomes in a set $R$, in which $(x, k) \in \varepsilon$ means that the agent considers the move $x$ to be `good' in the context $k$, mapping possible moves $x'$ to their outcome $k (x')$.

Selection relations are a variant of the better-known single-valued selection functions \cite{SequentialGames}, which are better behaved mathematically (forming a monad) but are less flexible as descriptions of agents. For example, the very common \emph{utility-maximizing agents} can be described as a multivalued selection function $\argmax$, where $(x, k) \in \argmax$ iff $x$ attains the global maximum of $k$.

\paragraph{The functor of selection relations.}
The starting observation of the following mathematical treatment of selection
relations is that a point $x \in X$ is exactly a state of $\diset{X}{R}$ in
the monoidal category $\Lens{\Set}$, while a function $k: X \to R$ is exactly a
costate of $\diset{X}{R}$.
Therefore a selection relation can be seen as a relation $\varepsilon \subseteq \Lens\Set(I,\diset{X}{R}) \times \Lens\Set(\diset{X}{R},I)$.

\begin{definition}
  Let $\Ma$ be a monoidal category. We define the functor $\Sel_{\Ma}: \Ma \to \Cat$ as follows:
  \begin{enumerate}
    \item $\Sel_{\Ma}(X)$ is the poset $\ps (\Ma(I,X) \times \Ma(X,I))$, ordered by inclusion. When $\varepsilon \in \Sel_{\Ma}(X)$, we write $\varepsilon(x,k)$ if $(x,k) \in \varepsilon$, emphasizing our view of $\varepsilon$ as a predicate.
    We call the elements of $\Sel_{\Ma}(X)$ \emph{selection relations on $X$}.
    \item For $f: X \to Y$, $E \in \Sel_{\Ma}(X)$, we define $\Sel_{\Ma}(f)(\varepsilon) \in \Sel_{\Ma}(Y)$ by
    $ \Sel_{\Ma} (f) (\varepsilon) = \{ (x \comp f, k) \mid \varepsilon (x, f \comp k) \} $
  \end{enumerate}
\end{definition}

We simply write $\Sel$ rather than $\Sel_{\Ma}$ when the category under consideration is obvious from context.
It is straightforward to verify that the $\Sel(f)$ are functors (i.e.\ monotone maps), that $\Sel (f\comp g) = \Sel(f) \comp \Sel(g)$, and $\Sel(1_{X}) = 1_{\Sel(X)}$.
To simplify notation, we write $f_{*}$ for $\Sel_{\Ma}(f)$.

Although we write this definition over an arbitrary monoidal category, we are mostly interested in the case where $\Ma = \Optic{\Ma'}$, where $\Ma'$ is a semicartesian symmetric monoidal category.
Then, equivalently, we have that $\Sel_{\Optic{\Ma'}} \diset{X}{X'} \cong \ps (\Ma' (I, X) \times \Ma' (X, X'))$.

%

\begin{example}
  Let $\Ma = \Lens\Set$. Then for any set $X$, there is a selection relation $\argmax_{X}$ on $\diset{X}{\Rb}$ defined by $\argmax_{X}(x,k)$ iff $k (x) \geq k (x')$ for all $x' \in X$.
  Analogous selection relations also exist over many other suitable categories.
\end{example}
%

\paragraph{The Nash product.}
There is a composition law for selection relations called the \emph{Nash product}, generalising a construction that appeared in \cite{hedges_backward_induction_repeated_games} as the ``sum of selection functions''.
Given a context on a tensor product $k: X \otimes Y \to I$, we can imagine two agents, one of whom controls the state $x: I \to X$ and one who controls the state $y: I \to Y$.
Suppose their decisions are governed respectively by the selection functions $\varepsilon$, $\delta$.
Given $y$, we can compose it with $k$ to form $k_y := (1_{X} \otimes y) \comp k : X \to I$, and ask whether the first player is satisfied with their choice in this context.
Analogously, we can ask whether the second player is satisfied with their choice in the context $k_x$ given by $x$ and $k$.
If both are satisfied, the composite state $x \otimes y: I \to X \otimes Y$ is said to be an \emph{equilibrium} --- nobody wants to unilaterally deviate from it.

In a general (non-cartesian) monoidal category, there may be states $s: I \to X \otimes Y$ which do not have the form $x \otimes y$.
We think of states of the form $x \otimes y$ as \emph{independent}. The Nash product models a situation where players make their choices without communication, and so we conservatively rule out any non-independent states as Nash equilibria.


\begin{proposition}
  The functor $\Sel: \Ma \to \Cat$ admits a laxator $\boxtimes: \Sel(X) \times \Sel(Y) \to \Sel(X \otimes Y)$
  given by
\[ \varepsilon \boxtimes \delta = \{ (x \otimes y, k) \mid \varepsilon (x, k_y) \text{ and } \delta (y, k_x) \} \]
where $k_x : Y \cong I \otimes Y \overset{x \otimes I}\longrightarrow X \otimes Y \overset{k}\longrightarrow I$ and $k_y : X \cong X \otimes I \overset{I \otimes y}\longrightarrow X \otimes Y \overset{k}\longrightarrow I$.
\end{proposition}
\begin{proof}
  It is straightforward to verify associativity and unitality of the lax monoidal structure,
  as well as commutativity. The nontrivial step is to verify that $\boxtimes$ is a natural transformation at all.
  To that end, let $f: X \to Y, f': X' \to Y'$, and let $\varepsilon \in \Sel(X), \delta \in \Sel(X')$.
  Then we must verify $f_{*}\varepsilon \boxtimes f'_{*}\delta = (f\otimes f')_{*}(\varepsilon \boxtimes \delta)$.

  This statement boils down to the fact that existential quantifiers commute.
  On the one hand, we have $(f \otimes f')_{*}(\varepsilon \boxtimes \delta) (s,k)$ if and only if there exists
  a factorization $s = t \comp (f \otimes f')$ and a further factorization $t = x \otimes x'$ so that
  $\varepsilon(x, (1_{x} \otimes x') \comp (f \otimes f') \comp k)$, and analogously for $x'$.
  On the other hand, we can ask that $s$ factor as $y \otimes y'$ so that $f_{*}\varepsilon(y, (1_{y} \otimes y') \comp k)$ (and analogously for the other one).
  This in turn means that $y$ factors as $x \comp f$, and that $y' = x' \comp f'$, with $\varepsilon(x, f \comp (1_{y} \otimes (x' \comp f')) \comp k)$,
  which, applying the equations of a symmetric monoidal category, is the same as the condition above.
  Hence the two selection relations are equivalent as desired.
\end{proof}

\begin{example}
  Let $X,Y$ be sets. Then $\argmax_{X} \boxtimes \argmax_{Y}$ is a relation between $X \times Y$ and ${X \times Y \to \Rb^2}$.
  Specifically, $(\argmax_{X} \boxtimes \argmax_{Y})((x,y),k)$ if and only if $(x,y)$ is a pure strategy Nash equilibrium for the 2-player normal form game with payoff matrix $k$.
\end{example}

Therefore, the functor $\Sel : \Ma \to \Cat$ is an indexed monoidal category.
Such an indexed category is turned into a (strong) monoidal functor $\int \Sel \to \Ma$ by the monoidal Grothendieck construction \cite{moeller_vasilakopoulou_monoidal_grothendieck}:

\begin{definition}
  The category $\Ma_\Sel := \int \Sel$ has
  \begin{enumerate}
    \item objects given by pairs $(X, \varepsilon)$ of an object of $\Ma$ together with a selection relation $\varepsilon \in \Sel (X)$ on it.
    \item morphisms $(X, \varepsilon) \to (Y, \delta)$ given by morphisms $f : X \to Y$ in $\Ma$ with the property that for all $h : I \to X$ and $k : Y \to I$, if $(h, f \comp k) \in \varepsilon$ then $(h \comp f, k) \in \delta$.
  \end{enumerate}
  It is monoidal with unit $(I, \top_I)$ and product $(X, \varepsilon) \otimes (Y, \delta) = (X \otimes Y, \varepsilon \boxtimes \delta)$.
  The projection $\pi : \Ma_\Sel \to \Ma$ is trivially strong monoidal.
\end{definition}



    \section{Open Games}

In this section we will equip parametrised optics with selection relations on the parameters. This results in a category in which we can do `compositional game theory', which refines open games \cite{hedges_etal_compositional_game_theory} by adding an explicit account of agents. More detail on this perspective can be found in \cite{other_paper}.



Let $\Ca$ be an $\Ma$-actegory. Since the forgetful functor $\pi : \Ma_\Sel \to \Ma$ is strong monoidal, $\Ca$ is also an $\Ma_\Sel$-actegory with action given by
$
	(X, \varepsilon) \action_\Sel Y := X \action Y
$.
We can then use this action to form the symmetric monoidal bicategory $\Para[\action_\Sel]{\Ca}$ as per Definition~\ref{def:para}.
In particular, morphisms of this category will be parametrised not just by objects of $\Ma$ but also by selection relations on them, which specify preferences on the parameter space.
%
Now let $\Ca$ and $\Da$ be monoidal $\Ma$-actegories. Then $\Optic{\Ca, \Da}$ is monoidal and instantiating the above construction for the action of $\ostar$ described in Proposition~\ref{prop:canon_optic_action} gives us a monoidal category $\Para[\ostar_\Sel]{\Optic{\Ca, \Da}}$. We think of this as a category of ``open games''. Although it is closely related to existing definitions of open games \cite{hedges_etal_compositional_game_theory,bolt_hedges_zahn_bayesian_open_games,atkey_etal_cgt_compositionally}, it differs in several key ways.


The parameter object $M$ is an object of $\Optic{\Ma}_\Sel$, hence a pair $\diset{\Omega}{\Comega}$,%
\footnote{We pronounce the symbol $\Comega$ as ``com\'ega''.}%
equipped with a selection relation $\varepsilon \in \Sel_{\Ma} \diset{\Omega}{\Comega}$. $\Omega$ is the object of `strategy profiles' found in other definitions of open games (usually called $\Sigma$). We refer to $\Comega$ as the set of `rewards', or `intrinsic utilities', that the agents playing a game actually optimise. Despite having such a clear game-theoretic reading, agents' rewards were left implicit in previous formulations of open games.

When $\Ca = \Da = \Ma$ is a semicartesian monoidal category, acting on itself by monoidal product, then a scalar (morphism $I \to I$) in $\Para[\otimes_\Sel]{\Optic{\Ma}}$ consists of a pair of parameter objects $\diset{\Omega}{\Comega}$, a selection relation $\varepsilon \subseteq \Ma (I, \Omega) \times \Ma (\Omega, \Comega)$, and a morphism $k : \Ma (\Omega, \Comega)$. We can then form the set $\{ \omega \in \Ma (I, \Omega) \mid \varepsilon (\omega, k) \}$. We think of this as the \emph{solution set} of the game: it is the set of strategy profiles that the agents accept in the context given by the game they are playing.

The simplest possible case is when $\Ca = \Da = \Ma = \Set$, acting on itself by cartesian product. In this case, a morphism $\diset{X}{S} \xrightarrow{(\Omega,\Comega)} \diset{Y}{R}$ consists of a pair of sets $\diset{\Omega}{\Comega}$, a selection relation ${\varepsilon \subseteq \Omega \times (\Omega \to \Comega)}$, and a lens $\Ga : \diset{\Omega}{\Comega} \otimes \diset{X}{S} \to \diset{Y}{R}$. This in turn consists of a ``play function'' $P : \Omega \times X \to Y$, and a function $\Omega \times X \times R \to \Comega \times S$ that splits into a ``coplay function'' ${C : \Omega \times X \times R \to S}$ and an ``intrinsic utility function'' $U : \Omega \times X \times R \to \Comega$.%
\footnote{In \cite{fong-spivak-tuyeras-backprop-as-functor} the function $U$ is called `update'.}
This data determines an open game in $\mathbf{OG} \big( \diset{X}{S}, \diset{Y}{R} \big)$, where $\mathbf{OG}$ is the usual category of open games in the sense of \cite{hedges_etal_compositional_game_theory}.
Specifically, we keep the set of strategy profiles, the play function and the coplay function the same.
It remains to choose an equilibrium function ${E : X \times (Y \to R) \to \mathcal P (\Omega)}$. We define it by ${E (h, k) = \{ \omega \mid \varepsilon (\omega, K_{h,k}) \}}$, where $K_{h,k}$ is given by the following composition in $\Lens\Set$:
\[ \diset{\Omega}{\Comega} \overset\cong\longrightarrow \diset{\Omega}{\Comega} \otimes I \xrightarrow{\diset{\Omega}{\Comega} \otimes h} \diset{\Omega}{\Comega} \otimes \diset{X}{S} \overset\Ga\longrightarrow \diset{Y}{R} \overset{k}\longrightarrow I \]
This construction defines a strong monoidal functor $\Para[\ostar_\Sel]{\Lens\Set} \to \mathbf{OG}$, and can also be carried out for Bayesian open games and other more general formulations of open games.

As a worked example, we will demonstrate how this formulation of game theory improves on open games by describing Prisoner's Dilemma and then modifying it so that both decisions are made by the same agent. For simplicity, our base will again be $\Set$ acting on itself by cartesian product, which describes deterministic games and pure strategy Nash equilibria. Given sets $X$ and $Y$, a decision to choose an element of $Y$ after observing an element of $X$ and with a real number payoff is described by a morphism $\diset{X}{1} \to \diset{Y}{\Rb}$ in $\Para{\Optic{\Set}}$, whose parameter sets are $\diset{X \to Y}{\Rb}$, where the play function is function application $(X \to Y) \times X \to Y$, and the intrinsic utility function is the projection $(X \to Y) \times X \times \Rb \to \Rb$. Note that we have separated the concept of a decision in a game, where a choice is made and a real number payoff is obtained, from the goal of the agent to maximise that payoff, or the mechanism by which they do so. This aligns with classical game theory, but is in contrast to open games, in which agents and decisions are conflated. In the case that $X$ has one element, describing a decision with no observation, decisions are particularly simple and can be denoted by a bending wire, as illustrated in Figure~\ref{fig:pd}.

\begin{figure}[ht]
	\begin{minipage}{0.5\textwidth}
		\sctikzfig{figures/pdargmax}
	\end{minipage}\begin{minipage}{0.5\textwidth}
		\sctikzfig{figures/pdpareto2}
	\end{minipage}
	\caption{(a) Standard Prisoner's Dilemma \hspace{3cm} (b) Modified Prisoner's Dilemma}\label{fig:pd}
\end{figure}

We take $Y = \{ C, D \}$ to be the set of moves in Prisoner's Dilemma. We take the tensor product of two decisions, describing that they are made in parallel. We then postcompose with the morphism $\diset{Y^2}{\Rb^2} \to I$ in $\Para[\times_\Sel]{\Lens{\Set}}$ given by lifting the payoff function $PD: Y^2 \to \Rb^2$ of Prisoner's Dilemma to a costate in $\Optic{\Set}$, taking the parameter sets to be $\diset{\Omega}{\Comega} = (1, 1) = I$. This results in a scalar $I \to I$ in $\Para[\times_\Sel]{\Optic{\Set}}$, with parameters $\diset{Y^2}{\Rb^2}$, described by the lower half of Figure~\ref{fig:pd}(a).
Ultimately, this determines a costate $\diset{Y^2}{\Rb^2} \to I$ in $\Lens\Set$, namely another copy of the payoff function $PD$.

In order to describe the standard Prisoner's Dilemma, we compose this diagram with a pair of $\argmax$ operators, describing the situation where a pair of players independently optimise their own payoff, without communicating. The lack of communication can be seen in the `air gap' in the top part of Figure~\ref{fig:pd}(a).%
\footnote{The top parts of Figure~\ref{fig:pd} are currently informal: we are drawing the selection functions as though they are states in $\Lens\Set$, but they are not. It may be possible, by thinking of them as `generalised states', to embed (perhaps by a clever use of co-Yoneda) into a larger category in which they are states; alternatively the top parts can be thought of as `just' a decoration of the top boundaries.}
The composition of a pair of $\argmax$ operators involves the Nash product, and in the end applying the costate $PD$ to the selection function $\argmax_Y \boxtimes \argmax_Y$ determines the solution set, namely the set $\{ (D, D) \}$, which is the unique Nash equilibrium of Prisoner's Dilemma.

Now consider the modified Prisoner's Dilemma denoted by Figure~\ref{fig:pd}(b).
This can be described in two different ways, which are equivalent but not the same: namely as Prisoner's Dilemma pulled back along an optic on parameters, coupled to the $\argmax$ selection function describing a classical optimising player; or alternatively as the ordinary Prisoner's Dilemma coupled to a selection function obtained by pushing forwards $\argmax$ along the same optic, describing a player who makes two choices to maximise the sum of two payoffs. Such a strategy profile is known as Hicks optimal, a strengthening of Pareto optimality.
For Prisoner's Dilemma, the solution set of this example is $\{ (C, C) \}$. Notably this is disjoint from the set of Nash equilibria.

Although this is only a very simple example, separating an open game into these `horizontal' and `vertical' parts provides a very flexible language for describing a wide variety of game-theoretic situations, such as those described in \cite{other_paper}.


    \bibliographystyle{eptcs}
    \bibliography{act21-cybercat}

\begin{thebibliography}{10}
\providecommand{\bibitemdeclare}[2]{}
\providecommand{\surnamestart}{}
\providecommand{\surnameend}{}
\providecommand{\urlprefix}{Available at }
\providecommand{\url}[1]{\texttt{#1}}
\providecommand{\href}[2]{\texttt{#2}}
\providecommand{\urlalt}[2]{\href{#1}{#2}}
\providecommand{\doi}[1]{doi:\urlalt{http://dx.doi.org/#1}{#1}}
\providecommand{\bibinfo}[2]{#2}

\bibitemdeclare{inproceedings}{WGAN}
\bibitem{WGAN}
\bibinfo{author}{Martin \surnamestart {Arjovsky}\surnameend},
  \bibinfo{author}{Soumith \surnamestart {Chintala}\surnameend} \&
  \bibinfo{author}{L{\'e}on \surnamestart {Bottou}\surnameend}
  (\bibinfo{year}{2017}): \emph{\bibinfo{title}{{Wasserstein GAN}}}.
\newblock In: {\sl \bibinfo{booktitle}{ICML'17: Proceedings of the 34th
  International Conference on Machine Learning}},
  \doi{10.5555/3305381.3305404}.

\bibitemdeclare{incollection}{atkey_etal_cgt_compositionally}
\bibitem{atkey_etal_cgt_compositionally}
\bibinfo{author}{Robert \surnamestart Atkey\surnameend}, \bibinfo{author}{Bruno
  \surnamestart Gavranovi\'c\surnameend}, \bibinfo{author}{Neil \surnamestart
  Ghani\surnameend}, \bibinfo{author}{Clemens \surnamestart Kupke\surnameend},
  \bibinfo{author}{J\'er\'emy \surnamestart Ledent\surnameend} \&
  \bibinfo{author}{Fredrik~Nordvall \surnamestart Forsberg\surnameend}
  (\bibinfo{year}{2020}): \emph{\bibinfo{title}{Compositional game theory,
  compositionally}}.
\newblock In: {\sl \bibinfo{booktitle}{Proceedings of \emph{Applied Category
  Theory 2020}}}, \bibinfo{publisher}{EPTCS}, \doi{10.4204/EPTCS.333.14}.

\bibitemdeclare{misc}{bolt_hedges_zahn_bayesian_open_games}
\bibitem{bolt_hedges_zahn_bayesian_open_games}
\bibinfo{author}{Joe \surnamestart Bolt\surnameend}, \bibinfo{author}{Jules
  \surnamestart Hedges\surnameend} \& \bibinfo{author}{Philipp \surnamestart
  Zahn\surnameend} (\bibinfo{year}{2019}): \emph{\bibinfo{title}{Bayesian open
  games}}.
\newblock \bibinfo{howpublished}{arXiv:1910.03656}.

\bibitemdeclare{inproceedings}{other_paper}
\bibitem{other_paper}
\bibinfo{author}{Matteo \surnamestart Capucci\surnameend},
  \bibinfo{author}{Neil \surnamestart Ghani\surnameend},
  \bibinfo{author}{J\'er\'emy \surnamestart Ledent\surnameend} \&
  \bibinfo{author}{Fredrik~Nordvall \surnamestart Forsberg\surnameend}
  (\bibinfo{year}{2021}): \emph{\bibinfo{title}{Translating extensive form
  games to open games with agency}}.
\newblock In: {\sl \bibinfo{booktitle}{Proceedings of Applied Category Theory
  2021}}, \bibinfo{publisher}{EPTCS}.
\newblock \bibinfo{note}{ArXiv:2105.06763}.

\bibitemdeclare{misc}{ProfunctorOptics}
\bibitem{ProfunctorOptics}
\bibinfo{author}{Bryce \surnamestart Clarke\surnameend}, \bibinfo{author}{Derek
  \surnamestart Elkins\surnameend}, \bibinfo{author}{Jeremy \surnamestart
  Gibbons\surnameend}, \bibinfo{author}{Fosco \surnamestart
  Loregian\surnameend}, \bibinfo{author}{Bartosz \surnamestart
  Milewski\surnameend}, \bibinfo{author}{Emily \surnamestart
  Pillmore\surnameend} \& \bibinfo{author}{Mario \surnamestart
  Rom\'an\surnameend} (\bibinfo{year}{2020}): \emph{\bibinfo{title}{Profunctor
  optics: {A} categorical update}}.
\newblock \bibinfo{howpublished}{arXiv:2001.07488}.

\bibitemdeclare{misc}{CatFoundGradDesc}
\bibitem{CatFoundGradDesc}
\bibinfo{author}{G.~S.~H. \surnamestart Cruttwell\surnameend},
  \bibinfo{author}{Bruno \surnamestart Gavranovi{\'c}\surnameend},
  \bibinfo{author}{Neil \surnamestart Ghani\surnameend}, \bibinfo{author}{Paul
  \surnamestart Wilson\surnameend} \& \bibinfo{author}{Fabio \surnamestart
  Zanasi\surnameend} (\bibinfo{year}{2021}): \emph{\bibinfo{title}{Categorical
  Foundations of Gradient-Based Learning}}.
\newblock \bibinfo{howpublished}{arXiv:2103.01931}.

\bibitemdeclare{article}{Dioptics}
\bibitem{Dioptics}
\bibinfo{author}{David \surnamestart Dalrymple\surnameend}
  (\bibinfo{year}{2019}): \emph{\bibinfo{title}{Dioptics: a Common
  Generalization of Open Games and Gradient-Based Learners}}.
\newblock {\sl \bibinfo{journal}{SYCO7}}.
\newblock
  \urlprefix\url{https://research.protocol.ai/publications/dioptics-a-common-generalization-of-open-games-and-gradient-based-learners/dalrymple2019.pdf}.

\bibitemdeclare{article}{SequentialGames}
\bibitem{SequentialGames}
\bibinfo{author}{Martin \surnamestart Escard\'o\surnameend} \&
  \bibinfo{author}{Paulo \surnamestart Oliva\surnameend}
  (\bibinfo{year}{2011}): \emph{\bibinfo{title}{Sequential games and optimal
  strategies}}.
\newblock {\sl \bibinfo{journal}{Proceedings of the Royal Society A}}
  \bibinfo{volume}{467}, pp. \bibinfo{pages}{1519--1545},
  \doi{10.1098/rspa.2010.0471}.

\bibitemdeclare{inproceedings}{fong-spivak-tuyeras-backprop-as-functor}
\bibitem{fong-spivak-tuyeras-backprop-as-functor}
\bibinfo{author}{Brendan \surnamestart Fong\surnameend}, \bibinfo{author}{David
  \surnamestart Spivak\surnameend} \& \bibinfo{author}{R\'emy \surnamestart
  Tuy\'eras\surnameend} (\bibinfo{year}{2019}): \emph{\bibinfo{title}{Backprop
  as functor: {A} compositional perspective on supervised learning}}.
\newblock In: {\sl \bibinfo{booktitle}{Proceedings of Logic in Computer Science
  (LiCS) 2019}}, \bibinfo{publisher}{ACM}, \doi{10.1109/LICS.2019.8785665}.

\bibitemdeclare{misc}{OpticCompositionAnimation}
\bibitem{OpticCompositionAnimation}
\bibinfo{author}{Bruno \surnamestart Gavranovi{\'c}\surnameend}
  (\bibinfo{year}{2021}): \emph{\bibinfo{title}{Animation of optic
  composition}}.
\newblock
  \bibinfo{howpublished}{\url{https://twitter.com/bgavran3/status/1366202140788731908}}.

\bibitemdeclare{misc}{OpticAnimation}
\bibitem{OpticAnimation}
\bibinfo{author}{Bruno \surnamestart Gavranovi{\'c}\surnameend}
  (\bibinfo{year}{2021}): \emph{\bibinfo{title}{Animation of optics}}.
\newblock
  \bibinfo{howpublished}{\url{https://twitter.com/bgavran3/status/1364644337968103428}}.

\bibitemdeclare{inproceedings}{hedges_etal_compositional_game_theory}
\bibitem{hedges_etal_compositional_game_theory}
\bibinfo{author}{Neil \surnamestart Ghani\surnameend}, \bibinfo{author}{Jules
  \surnamestart Hedges\surnameend}, \bibinfo{author}{Viktor \surnamestart
  Winschel\surnameend} \& \bibinfo{author}{Philipp \surnamestart
  Zahn\surnameend} (\bibinfo{year}{2018}): \emph{\bibinfo{title}{Compositional
  game theory}}.
\newblock In: {\sl \bibinfo{booktitle}{Proceedings of Logic in Computer Science
  (LiCS) 2018}}, \bibinfo{publisher}{ACM}, pp. \bibinfo{pages}{472--481},
  \doi{10.1145/3209108.3209165}.

\bibitemdeclare{misc}{GAN}
\bibitem{GAN}
\bibinfo{author}{Ian~J. \surnamestart Goodfellow\surnameend},
  \bibinfo{author}{Jean \surnamestart Pouget-Abadie\surnameend},
  \bibinfo{author}{Mehdi \surnamestart Mirza\surnameend}, \bibinfo{author}{Bing
  \surnamestart Xu\surnameend}, \bibinfo{author}{David \surnamestart
  Warde-Farley\surnameend}, \bibinfo{author}{Sherjil \surnamestart
  Ozair\surnameend}, \bibinfo{author}{Aaron \surnamestart Courville\surnameend}
  \& \bibinfo{author}{Yoshua \surnamestart Bengio\surnameend}
  (\bibinfo{year}{2014}): \emph{\bibinfo{title}{Generative Adversarial
  Networks}}, \doi{10.1145/3422622}.

\bibitemdeclare{inproceedings}{hedges_backward_induction_repeated_games}
\bibitem{hedges_backward_induction_repeated_games}
\bibinfo{author}{Jules \surnamestart Hedges\surnameend} (\bibinfo{year}{2018}):
  \emph{\bibinfo{title}{Backward induction for repeated games}}.
\newblock In: {\sl \bibinfo{booktitle}{Proceedings of {M}athematically
  {S}tructured {F}unctional {P}rogramming ({MSFP}) 2018}}, {\sl
  \bibinfo{series}{Electronic Proceedings in Theoretical Computer Science}}
  \bibinfo{volume}{275}, pp. \bibinfo{pages}{35--52},
  \doi{10.4204/EPTCS.275.5}.

\bibitemdeclare{inproceedings}{HigherOrderDecisions}
\bibitem{HigherOrderDecisions}
\bibinfo{author}{Jules \surnamestart Hedges\surnameend}, \bibinfo{author}{Paulo
  \surnamestart Oliva\surnameend}, \bibinfo{author}{Evguenia \surnamestart
  Shprits\surnameend}, \bibinfo{author}{Viktor \surnamestart
  Winschel\surnameend} \& \bibinfo{author}{Philipp \surnamestart
  Zahn\surnameend} (\bibinfo{year}{2017}): \emph{\bibinfo{title}{Higher-order
  decision theory}}.
\newblock In \bibinfo{editor}{J\"org \surnamestart Rothe\surnameend}, editor:
  {\sl \bibinfo{booktitle}{Algorithmic Decision Theory}}, {\sl
  \bibinfo{series}{Lecture Notes in Artificial Intelligence}}
  \bibinfo{volume}{10576}, \bibinfo{publisher}{Springer}, pp.
  \bibinfo{pages}{241--254}, \doi{10.1007/978-3-319-67504-6\_17}.

\bibitemdeclare{inproceedings}{hedges_etal_selection_equilibria_higher_order_games}
\bibitem{hedges_etal_selection_equilibria_higher_order_games}
\bibinfo{author}{Jules \surnamestart Hedges\surnameend}, \bibinfo{author}{Paulo
  \surnamestart Oliva\surnameend}, \bibinfo{author}{Evguenia \surnamestart
  Shprits\surnameend}, \bibinfo{author}{Viktor \surnamestart
  Winschel\surnameend} \& \bibinfo{author}{Philipp \surnamestart
  Zahn\surnameend} (\bibinfo{year}{2017}): \emph{\bibinfo{title}{Selection
  equilibria of higher-order games}}.
\newblock In: {\sl \bibinfo{booktitle}{Practical aspects of declaritive
  languages}}, {\sl \bibinfo{series}{Lecture Notes in Computer Science}}
  \bibinfo{volume}{10137}, \bibinfo{publisher}{Springer}, pp.
  \bibinfo{pages}{136--151}, \doi{10.1007/978-3-319-51676-9\_9}.

\bibitemdeclare{book}{johnson2021}
\bibitem{johnson2021}
\bibinfo{author}{Niles \surnamestart Johnson\surnameend} \&
  \bibinfo{author}{Donald \surnamestart Yau\surnameend} (\bibinfo{year}{2021}):
  \emph{\bibinfo{title}{2-dimensional Categories}}.
\newblock \bibinfo{publisher}{Oxford University Press, USA},
  \doi{10.1093/oso/9780198871378.001.0001}.

\bibitemdeclare{book}{loregian2020coend}
\bibitem{loregian2020coend}
\bibinfo{author}{F.~\surnamestart Loregian\surnameend} (\bibinfo{year}{2021}):
  \emph{\bibinfo{title}{Coend Calculus}}, \bibinfo{edition}{first} edition.
\newblock {\sl \bibinfo{series}{London Mathematical Society Lecture Note
  Series}} \bibinfo{volume}{468}, \bibinfo{publisher}{Cambridge University
  Press}, \doi{10.1017/9781108778657}.
\newblock \bibinfo{note}{ISBN 9781108746120}.

\bibitemdeclare{article}{moeller_vasilakopoulou_monoidal_grothendieck}
\bibitem{moeller_vasilakopoulou_monoidal_grothendieck}
\bibinfo{author}{Joe \surnamestart Moeller\surnameend} \&
  \bibinfo{author}{Christina \surnamestart Vasilakopoulou\surnameend}
  (\bibinfo{year}{2020}): \emph{\bibinfo{title}{Monoidal {G}rothendieck
  construction}}.
\newblock {\sl \bibinfo{journal}{Theory and applications of categories}}
  \bibinfo{volume}{35}(\bibinfo{number}{31}), pp. \bibinfo{pages}{1159--1207}.

\bibitemdeclare{unpublished}{meyers_string_diagrams_double_categories}
\bibitem{meyers_string_diagrams_double_categories}
\bibinfo{author}{David~Jaz \surnamestart Myers\surnameend}
  (\bibinfo{year}{2018}): \emph{\bibinfo{title}{String diagrams for double
  categories and equipments}}.
\newblock \bibinfo{note}{ArXiv:1612.02762}.

\bibitemdeclare{unpublished}{myers2020double}
\bibitem{myers2020double}
\bibinfo{author}{David~Jaz \surnamestart Myers\surnameend}
  (\bibinfo{year}{2020}): \emph{\bibinfo{title}{Double Categories of Open
  Dynamical Systems}}.
\newblock \bibinfo{note}{ArXiv:2005.05956}.

\bibitemdeclare{book}{cts}
\bibitem{cts}
\bibinfo{author}{David~Jaz \surnamestart Myers\surnameend}
  (\bibinfo{year}{2021}): \emph{\bibinfo{title}{Categorical Systems Theory}}.
\newblock
  \urlprefix\url{https://github.com/DavidJaz/DynamicalSystemsBook/tree/master/book}.

\bibitemdeclare{unpublished}{CategoriesOfOptics}
\bibitem{CategoriesOfOptics}
\bibinfo{author}{Mitchell \surnamestart Riley\surnameend}
  (\bibinfo{year}{2018}): \emph{\bibinfo{title}{Categories of optics}}.
\newblock \bibinfo{note}{ArXiv:1809.00738}.

\bibitemdeclare{misc}{RileyZulip}
\bibitem{RileyZulip}
\bibinfo{author}{Mitchell \surnamestart Riley\surnameend} \&
  \bibinfo{author}{The CT~Zulip \surnamestart Community\surnameend}
  (\bibinfo{year}{2021}): \emph{\bibinfo{title}{An involution on `Learn'}}.
\newblock
  \urlprefix\url{https://mattecapu.github.io/ct-zulip-archive/stream/229156-practice:-applied-ct/topic/An.20involution.20on.20'Learn'.html#235034768}.

\bibitemdeclare{article}{vagner2014algebras}
\bibitem{vagner2014algebras}
\bibinfo{author}{Dmitry \surnamestart Vagner\surnameend},
  \bibinfo{author}{David~I \surnamestart Spivak\surnameend} \&
  \bibinfo{author}{Eugene \surnamestart Lerman\surnameend}
  (\bibinfo{year}{2014}): \emph{\bibinfo{title}{Algebras of open dynamical
  systems on the operad of wiring diagrams}}.
\newblock {\sl \bibinfo{journal}{arXiv preprint arXiv:1408.1598}}.

\bibitemdeclare{article}{vasilakopoulou2018enriched}
\bibitem{vasilakopoulou2018enriched}
\bibinfo{author}{Christina \surnamestart Vasilakopoulou\surnameend}
  (\bibinfo{year}{2018}): \emph{\bibinfo{title}{On enriched fibrations}}.
\newblock {\sl \bibinfo{journal}{arXiv preprint arXiv:1801.01386}}.

\bibitemdeclare{book}{strategy}
\bibitem{strategy}
\bibinfo{author}{Joel \surnamestart Watson\surnameend} (\bibinfo{year}{2013}):
  \emph{\bibinfo{title}{Strategy: {A}n introduction to game theory}}.
\newblock \bibinfo{publisher}{W. W. Norton}.

\end{thebibliography}
\end{document}